\documentclass[10pt]{amsart}
\usepackage{amssymb}
\usepackage{bm}
\usepackage{graphicx}
\usepackage[centertags]{amsmath}
\usepackage{amsfonts}
\usepackage{amsthm}
\newtheorem{thm}{Theorem}
\newtheorem{cor}[thm]{Corollary}
\newtheorem{lem}[thm]{Lemma}

\newtheorem{prop}[thm]{Proposition}
\newtheorem{claim}[thm]{Claim}

\newtheorem{defn}[thm]{Definition}
\theoremstyle{definition}

\newtheorem{rem}{Remark}
\newtheorem{ques}{Question}

\newtheorem{notation}{Notation}

\newcommand{\nn}{\mathbb{N}}
\newcommand{\qq}{\mathbb{Q}}
\newcommand{\ee}{\varepsilon}

\newcommand{\bt}{\nn^{<\nn}}

\newcommand{\seg}{\mathfrak{s}}

\newcommand{\aaa}{\mathcal{M}}

\newcommand{\ttt}{\mathcal{T}}

\newcommand{\supp}{\mathrm{supp}}
\newcommand{\range}{\mathrm{range}}
\newcommand{\sspan}{\mathrm{span}}

\begin{document}

\title{On spaces admitting no $\ell_p$ or $c_0$ spreading model}
\author{Spiros A. Argyros and Kevin Beanland}

\address{National Technical University of Athens, Faculty of Applied
Sciences, Department of Mathematics, Zografou Campus, 157 80, Athens, Greece.}
\email{saargyros@math.ntua.gr}

\address{Department of Mathematics and Applied Mathematics, Virginia Commonwealth
University, Richmond, VA 23284.}
\email{kbeanland@vcu.edu}

\thanks{The second author acknowledges support from the Fulbright Foundation - Greece and the Fulbright
program.}

\maketitle


\begin{abstract}
It is shown that for each separable Banach space $X$ not admitting
$\ell_1$ as a spreading model there is a space $Y$ having $X$ as a quotient
and not admitting any $\ell_p$ for $1 \leq p < \infty$ or $c_0$ as a spreading model.

We also include the solution to a question of W.B. Johnson and H.P. Rosenthal on the existence
of a separable space not admitting as a quotient any space with separable dual.
\end{abstract}

\section{Introduction}

A Banach space $X$ is said to have, for $1 \leq p < \infty$, an $\ell_p$ spreading model if
there is a $\delta >0$ and a sequence $(x_n)$ in $X$
such that for all $n \in \nn$, $n\leq \ell_1 < \cdots < \ell_n$ and $(a_i)_i \in c_{00}$,
$$\delta \|(a_i)_{i=1}^n\|_p  \leq \|\sum_{i=1}^n a_i x_{\ell_i} \| \leq  \frac{1}{\delta} \|(a_i)_{i=1}^n\|_p.$$
For $p=\infty$ we say $X$ has a $c_0$ spreading model.  The first example of a space not admitting any $\ell_p$ or $c_0$
as a spreading model was provided by E. Odell and Th. Schlumprecht in \cite{OSKrivine}.  This space $X_S$ is the completion of $c_{00}(\nn)$ under
a norm that is a modification of the norm of Schlumprecht's space $S$.  As with the norming set of $S$, the norming set of $X_S$ is
defined using the saturation method. In the case of $X_S$, the norming set includes $\ell_2$ convex combination of certain
weighted functionals at every step of its, inductive, construction.  The idea of including this type of structure in
a given norming set can be
traced back to work of R.C. James \cite{JamesUncond} and can also be found in the W.T. Gowers' construction \cite{GowersTree}
of a space not containing $c_0$, $\ell_1$ or a reflexive subspace.
Recently, in \cite{AMP}, it was shown that there exist hereditarily indecomposable spaces not
admitting any $\ell_p$ or $c_0$ as a spreading model.  In \cite{AOST}, the authors construct
a space not admitting an $\ell_p$, $c_0$ or reflexive spreading model.  In paper \cite{AKTkspreading} they
show that a variant of the space $X_S$ does not admit any $\ell_p$ or $c_0$ as a $k$-iterated
spreading model for any $k \in \nn$.

In \cite{ATMemoirs} it is shown that every separable Banach space either contains $\ell_1$ or is a quotient of
a hereditarily indecomposable space. The main theorem of this paper is a similar dichotomy
for spaces that do not admit $\ell_1$ spreading models.
By the well-known lifting property of $\ell_1$, if a space $X$ admits an $\ell_1$ spreading model, then any $Y$ having
$X$ as a quotient must also admit an $\ell_1$ spreading model.   More precisely, our main theorem is the following
dichotomy.

\begin{thm}
Let $X$ be a separable Banach space.  Exactly one of the following holds:
\begin{enumerate}
\item $X$ admits an $\ell_1$ spreading model.
\item There is a separable space $Y$ not admitting any $\ell_p$ for $1\leq p <\infty$ or $c_0$ as a spreading model such that $X$ is a quotient of $Y$.
\end{enumerate}
\label{dichotomy}
\end{thm}

We outline the proof of the above theorem:  The first step is to pass from a separable
space $X$ not admitting an $\ell_1$ spreading model to a space $Z_X$ with a bimonotone Schauder basis, having $X$ as a quotient and not having
an $\ell_1$ spreading model.  The second step is to show that for any space $Z$ with a bimonotone Schauder basis
and not having an $\ell_1$ spreading model, one can construct a {\em ground set} $G_Z \subset c_{00}$ such that the space $Y_{G_Z}$,
having $G_Z$ as its norming set, also does not have $\ell_1$ as a spreading model.  After this, using the method in
\cite{OSKrivine}, we construct a space $T_{G_Z,2}$ not having any $\ell_p$ or $c_0$ as a spreading model.
The final, and most difficult, step is to show that the space $T_{G_Z,2}$ has $Z$ as a quotient.

The paper is organized as follows.  In section 2 we give several definitions including the definition
of a ground set $G_Z$ determined by a space with a basis $Z$.
We also prove the first two steps stated above.  In section 3
we define, for a space $Z$ with a basis, the space $T_{G_Z,2}$ and show that it does not admitting any $\ell_p$ or $c_0$
spreading model.  In section
4 we prove that $T_{G_Z,2}$ has $Z$ as a quotient.  We conclude by combining
the above to prove our main result and showing that if a space $X$ has as a quotient every space not
admitting an $\ell_1$ spreading model, then $X$ contains $\ell_1$.

The final section includes a result that is independent from the rest of the paper.  
Namely, we observe that a space constructed in \cite{AAT} does not admit as quotient any space with separable dual.
This solves a question posed in \cite[page 86, Remark IV.1]{JR}. We thank W.B. Johnson for
bringing this problem to our attention and simplifying our original solution.

\section{Spaces having no $\ell_1$ spreading model}

Let $c_{00}$ be the vector space of all finitely supported scalar sequences and $(e_n)$ denote
the unit vector basis of $c_{00}$.
Suppose $X$ has a Schauder basis $(x_n)_{n \in \nn}$. Let $(x_n^*)$ be the biorthogonal functionals
of $(x_n)$.  For $x \in \sspan{(x_i)_{i=1}^\infty}$ let $\supp(x) =\{ i : x^*_i(x) \not=0\}$.  Let $B_X=\{x \in X: \|x\|\leq 1\}$
and $S_X =\{x \in X: \|x\|=1\}$.

Our first definition can be found in \cite[Definition 14.1]{ATMemoirs}.

\begin{defn}
Let $Z$ be a space with a bimonotone Schauder basis $(z_i)_{i \in \nn}$ and $(\Lambda_i)_{i \in \nn}$ be a partition
of $\nn$ such that each $\Lambda_i$ is infinite.  Define  $G_Z \subset c_{00}$ as follows:
\begin{equation}
\begin{split}
G_Z =\{ \sum_{i=1}^d a_i (\sum_{n \in E \cap \Lambda_i} e_n^*) : &~ (a_i)_{i=1}^d \subset \qq,~ \|\sum_{i =1}^d a_i z^*_i \| \leq 1\\
& ~E \mbox{ finite interval of } \nn\}.
\label{groundZ}
\end{split}
\end{equation}
$G_Z$ is an example of a {\em ground set}.  Let
$Y_{G_Z}$ be the Banach space that is the completion of $c_{00}$ with the norming set $G_Z$ and $(y_n)$ denote
it natural basis.
\end{defn}

The space $Z$ is naturally a quotient of $Y_{G_Z}$.  In the next definition, we define the map.  In Proposition \ref{manyitems}
we will show it is a quotient map.

\begin{defn}
Define $Q_{G_Z}: Y_{G_Z} \to Z$ by
\begin{equation}
Q_{G_Z}y_n = e_i \mbox{ for }n \in \Lambda_i.
\end{equation}
Notice that for each $i \in \nn$ and $(a_j)_{j \in \Lambda_i}$ we have
\begin{equation}
Q_{G_Z} (\sum_{j \in \Lambda_i} a_j y_j) = (\sum_{j \in \Lambda_i }a_j)e_i. \label{imageQ}
\end{equation}
\label{QZ}
\end{defn}

For an arbitrary separable space $X$ we can construct a space $Z_X$ with a basis that retains many properties of $X$.  The
following construction can be found in \cite{ATMemoirs} (also see \cite{Schechtman}).
\begin{defn}
Let $X$ be a separable Banach space.  Let $R: \ell_1 \to X$ be a bounded linear operator such that
$(Re_n)_{n =1}^\infty$ is a dense subset of $S_X$.  Let
$$W= \{ ER^*x^* : x^* \in B_X~ \mbox{ and $E$ is an interval of $\nn$}\}.$$
Define the following norm on $c_{00}$:  For $(a_i) \in c_{00}$ let
\begin{equation}
\begin{split}
\|\sum_i a_i e_i \|_{Z_X} & = \sup\{ f(\sum_i a_i e_i) : f \in W\}\\
& =  \sup \{ \sum_{i \in E} a_i x_i : E \mbox{ finite interval in } \nn\}
\label{basisquotient}
\end{split}
\end{equation}
In the above $Re_i=x_i$ for all $i \in \nn$.
Let $Z_X$ be the completion of $c_{00}$ with the above norm.
\end{defn}

Note that $Z_X$ depends on the choice of the dense sequence $(x_n)$. Note
that $(e_n)$ is a bimonotone Schauder basis of $Z_X$.  We now define
the natural mapping from $Z_X$ to $X$.  It is easy to see that this map is a bounded quotient map.

\begin{defn}
Let $X$ be a separable Banach space such that $(x_i)$ is dense in $S_X$ and $Z_X$ be defined
as above.  Define $Q_X: Z_X \to X$ by $Q_X(e_i) = x_i$ and extending linearly.
\end{defn}

In the next proposition we collect some important facts concerning the spaces
and operators defined above.  The proofs can be found in \cite[Lemmas 14.3 and 14.8]{ATMemoirs}.

\begin{prop}
Let $X$ be a separable space and $Z$ be a space with a bimonotone Schauder basis. Then
\begin{itemize}
\item[(1)]  For every $x\in S_X$ there is a $y \in S_{Z_X}$ such that $Q_Xy=x$.  In particular,
$Q_X: Z_X \to X$ is a quotient.
\item[(2)] If $X$ does not contain $\ell_1$ then $Z_X$ does not contain $\ell_1$.
\item[(3)] For every $z\in S_Z$ there is a $y \in S_{Y_{G_Z}}$ such that $Q_{G_Z}y=x$.
In particular, $Q_{G_Z} : Y_{G_Z} \to Z$ is a quotient.
\item[(4)] If $Z$ does not contain $\ell_1$ then $Y_{G_Z}$ does not contain $\ell_1$.
\end{itemize}
\label{manyitems}
\end{prop}

\begin{proof}
We prove only (3). Let $z= \sum_{i=1}^d a_i z_i \in Z$ with $\|z\|=1$.
Let $\ell_i \in \Lambda_i$ for all $i = 1, \ldots, d$ and $x=\sum_{i=1}^d a_i e_{\ell_i}$. Clearly
$Q_{G_Z}x=z$.   We will show that $\|x\|=1$.

  Let $(b_i)_{i=1}^d$ such that $\| \sum_{i=1}^d b_i z_i^* \| \leq 1$
and $\sum_{i=1}^d a_i b_i = (\sum_{i=1}^d b_i z_i^*) (\sum_{i=1}^d a_i z_i) =1$.  By definition
$ \sum_{i=1}^d b_i e^*_{\ell_i} \in G_Z$.  Therefore
$$1= (\sum_{i=1}^d b_i e^*_{\ell_i} )(\sum_{i=1}^d a_i e_{\ell_i}) \leq \|x\|.$$

Let $\ee>0$. Find scalars $(c_i)_{i =1}^d$ and an interval $E$ such that
$$\sum_{i \in E}c_i a_i = (\sum_{i \in E} c_i e^*_{\ell_i})( \sum_{i=1}^d a_i e_{\ell_i}) \geq \frac{\|x\|}{(1+\ee)}.$$
Using bimonotonicity $\| \sum_{i \in E} c_i z_i^* \| \leq 1$.  Therefore
$$ \frac{\|x\|}{(1+\ee)} \leq \sum_{i \in E}c_i a_i \leq (\sum_{i \in E} c_i z^*_{i})( \sum_{i=1}^d a_i z_i) \leq 1.$$
Since $\ee$ was arbitrary $\|x\| \leq 1$.
\end{proof}

Our next result of this section is the following analogue of Proposition \ref{manyitems} (4).

\begin{prop}
If $Z$ has a basis and does not admit an $\ell_1$ spreading model then $Y_{G_Z}$ does not admit an $\ell_1$ spreading model. \label{nosminground}
\end{prop}
Before proving the above, we make a remark that allows us to estimate the norms of vectors in $Y_{G_Z}$ in terms of there images under the quotient map $Q_{G_Z}$.  We also recall an important theorem on the existence of $\ell_1$ spreading models in a Banach space
not containing $\ell_1$.
\begin{rem}
Let $\sum_j a_j e_j \in Y_{G_Z},$ then
\begin{equation}
\|\sum_j a_j e_j \|_{G_Z} \leq \sup \{ \| Q_{G_Z} P_E ( \sum_{ j } a_j e_j )\|_Z: E \mbox{ is an interval in }\nn\}.
\end{equation}
In the above, $P_E(\sum_j a_j e_j)= \sum_{j\in E} a_j e_j$.
\label{firstintervalestimate}
\end{rem}
\begin{proof}
Let $E \subset \nn$ be an interval and $(b_i)_{i=1}^d$ be scalars such that $\| \sum_{i=1}^d b_i z^*_i \| \leq 1$.  Using (\ref{imageQ})
\begin{equation}
\begin{split}
|(\sum_{i=1}^d b_i \sum_{j \in E\cap \Lambda_i} e^*_j )(\sum_{i =1}^\infty \sum_{j \in \Lambda_j} a_j e_j)|& = | \sum_{i=1}^d b_i ( \sum_{j \in \Lambda_i \cap E} a_j )| \\
& = | ( \sum_{i=1}^d b_i e^*_i ) ( \sum_{i=1}^d (\sum_{i \in \Lambda_j \cap E} a_j ) z_i) | \\
&  \leq \|  Q_{G_Z}\sum_{i=1}^d ( \sum_{j \in \Lambda_i \cap E} a_j z_j) \|_Z \\
& =  \|  Q_{G_Z} P_E ( \sum_{j} a_j z_j) \|_Z
\end{split}
\end{equation}
Since $E$ and $(b_i)_{i=1}^d$ are arbitrary, the remark follows.
\end{proof}
The next theorem we need  due to H.P. Rosenthal \cite{RosenthalSM}.  A similar statement can be found in \cite{BeauzamyLSM}.
\begin{thm}
Let $X$ be a Banach space not containing $\ell_1$.  Then the following are equivalent.
\begin{enumerate}
\item $X$ does not have an $\ell_1$ spreading model.
\item Every seminormalized weakly null sequence $(x_n)$ has a Cesaro summable subsequence.  In other words,
there is a subsequence $(y_n)$ of $(x_n)$ such that $\|1/n\sum_{i=1}^n y_i\| \to 0$ as $n \to \infty$.
\end{enumerate}
\label{Cesaro}
\end{thm}
\begin{proof}[Proof of Proposition \ref{nosminground}]
Using Proposition \ref{manyitems} (4), $Y_{G_Z}$ does not contain $\ell_1$.
Let $(z_n)$ be a seminormalized weakly null sequence in $Y_{G_Z}$.
Our goal is to extract a Cesaro summable subsequence.  We pass to a subsequence of $(z'_n)$ of $(z_n)$ that
has the following properties:
\begin{enumerate}
\item $(z'_n)$ is equivalent to a block sequence of $(y_n)$;
\item $(Q_{G_Z}z'_n)$ is either a bimonotone basic sequence or $\|Q_{G_Z}z'_n\|<2^{-n}$;
\item $(Q_{G_Z}z'_n)$ is Cesaro summable.
\end{enumerate}
Notice that $(2)$ has two cases.  Let $\ee>0$.  Find $n_0$ such that $\|1/n_0 \sum_{i=1}^{n_0} Q_{G_Z} z^\prime_n \| +3/n_0 < \ee$.
Let $E$ be an arbitrary interval.  Find $n_1, n_2$ in $\nn$ such that
\begin{equation*}
\begin{split}
& n_1 = \min\{ n \in \{ 1, \ldots, n_0\} : \supp z^\prime_n \cap E \not= \emptyset \},\\
& n_2 = \max\{ n  \in \{ 1, \ldots, n_0\}: \supp z^\prime_n \cap E \not= \emptyset \}
\end{split}
\end{equation*}
Assume first that $(Q_{G_Z}z_n^\prime)$ is bimonotone basic.  Since $(z^\prime_n)$ is a block, for $n_1< n < n_2$ we have $Q_{G_Z}P_E(z^\prime_n)=Q_{G_Z}(z^\prime_n)$.  Using this fact, our assumption on $n_0$ and fact that $(Q_{G_Z}z_n^\prime)$ is basic we have
\begin{equation}
\begin{split}
\|Q_{G_Z} P_E (\frac{1}{n_0} \sum_{n=1}^{n_0} z^\prime_n ) \| &= \frac{1}{n_0}\bigg[ \| Q_{G_Z} P_E(z^\prime_{n_1})\| + \| \sum_{n=n_1+1}^{n_2-1} Q_{G_Z} z^\prime_n\| + \|Q_{G_Z} P_E z^\prime_{n_2}\| \bigg] \\
& \leq \frac{2}{n_0} + \|\frac{1}{n_0} \sum_{i=1}^{n_0} Q z^\prime_n \| < \ee. \label{Qestimates}
\end{split}
\end{equation}
Since $E$ was arbitrary, applying Remark \ref{firstintervalestimate}, we finish the proof in the case when $(Q_{G_Z}z_n^\prime)$ bimonotone basic.
In the case that $\|Q_{G_Z}z^\prime_n\|<2^{-n}$ we have
 \begin{equation}
 \frac{1}{n_0} \|\sum_{n=n_1+1}^{n_2-1} Q_{G_Z} z^\prime_n\| < \frac{1}{n_0}.
\label{tozero}
 \end{equation}
Proceeding the same way as in the first inequality of (\ref{Qestimates}), using (\ref{tozero}) and the fact that $3/n_0<\ee$, we finish the proof.
\end{proof}

The final proposition of this section is analogous to Proposition \ref{manyitems} (2).

\begin{prop}
If $X$ does admit an $\ell_1$ spreading model then $Z_X$ does not admit an $\ell_1$ spreading model. \label{separabletobasis}
\end{prop}

\begin{proof}
By Proposition \ref{manyitems} (b) we have that $Z_X$ does not contain $\ell_1$.  Therefore, applying Theorem \ref{Cesaro} we can consider an arbitrary seminormalized weakly null sequence and show it has a
 Cesaro summable subsequence.  The following remark is a restatement of (\ref{basisquotient}).
\begin{rem}
Let $\sum_i a_i e_i \in Z_X$ and $P_E (\sum_i a_i e_i)  = \sum_{i \in E} a_i e_i$.  Then
$$\| \sum_i a_i e_i \|_Z = \sup \{ \| Q_X P_E ( \sum_i a_i e_i )\|_X: E \mbox{ is an interval in }\nn\}.$$
\label{intervalestimate}
\end{rem}
For an arbitrary seminormalized weakly null sequence in $Z_X$ we can pass to a subsequence satisfying
the same (1), (2) and (3) as in the proof of Proposition \ref{nosminground}.  Since Remark \ref{intervalestimate} is the same
are Remark \ref{firstintervalestimate} with a different quotient map, by mimicking the proof of Proposition \ref{nosminground} it can be shown that this subsequence in Cesaro summable, as required.
\end{proof}

\section{The construction of $T_{G_Z,2}$ and some properties}

For the rest of the paper we fix a space $Z$ having a bimonotone Schauder basis and not admitting
an $\ell_1$ spreading model.  In this section we define the space
$T_{G_Z,2}$ that does not admit any $\ell_p$ or $c_0$ spreading model
and has $Z$ as a quotient.  To star,t we fix two
increasing sequences of natural numbers $(m_j)_{j=1}^\infty$ and $(n_j)_{j=1}^\infty$
satisfying:
\begin{itemize}
\item[(a)] $\sum_{i=1}^\infty \frac{1}{m_i} < \frac{1}{10}$.
\item[(b)] $\lim_{i \to \infty} \frac{((i-1)n_{i-1})^{s_i}}{n_i} = 0$.  Where $s_i =\log_{m_1}(m_i)$.
\item[(c)] $\lim_{i \to \infty} \frac{n_i^\alpha}{m_i} = \infty$ for all $\alpha >0$.
\end{itemize}
We now define the norming set
inductively.
Let $G_0=G_Z$ (recall the definition from (\ref{groundZ})). Suppose $G_n$ has been defined for some $n \geq 0$  define $G_{n+1}$ as follows:
\begin{equation*}
 G_{n+1}^\prime=\{ \frac{1}{m_j} \sum_{i=1}^d f_i : j \in \nn,~ d \leq n_j, ~(f_i)_{i=1}^d \subset G_n \mbox{ and } f_1 < \cdots < f_d \}
\end{equation*}
\begin{equation*}
G_{n+1}^{\prime\prime}=\{ \sum_{i=1}^n \lambda_i f_i : n \in \nn,~\lambda_i\geq 0,~\sum_{i=1}^n \lambda^2_i\leq 1, ~(f_i)_{i=1}^n \subset G_{n+1}^\prime, w(f_i)=m_{i}\}
\end{equation*}
Let $G_{n+1} =G_{n+1}^{\prime \prime} \cup G_n$.  Let $D_{G_Z}= \cup_{n=1}^\infty G_n$.

Let $T_{G_Z,2}$ be the completion of $c_{00}$ under the norm $\| x \|_{D_{G_Z}} = \{ f(x) : f \in D_G\}$.
\begin{notation}
Let $f \in G_n\setminus G_Z$  for some $n \in \nn$.
\begin{enumerate}
\item If $f \in G_{n}^\prime$ then $f= 1/m_j \sum_{i=1}^d f_i$ for some $j \in \nn$. In this case we say $f$ is weighted and set the `weight of $f$' $=w(f) = m_j$.  Note that this weight is not unique.
\item If $f \in G_{n}^{\prime\prime}$ then $f= \sum_{i=1}^k \lambda_i f_i$ where $w(f_i)=m_i$ and $\sum_{i=1}^k \lambda_i^2 \leq 1$.  Set
$w(f)=\{ m_i : \lambda_i \not=0\}$. If $|w(f)|>1$ we say $f$ is not weighted.
\item For $\sum_{i=1}^k \lambda_i f_i  \in G_{n}^{\prime\prime}$
let $f_{\leq i_0} = \sum_{i=1}^{i_0} \lambda_i f_i$ and $f_{>i_0}= \sum_{i=i_0+1}^{k} \lambda_i f_i$.
\end{enumerate}
\end{notation}

A variant of the next theorem can be found in \cite[Theorem 11.3]{AMP}.  We include the proof here to give a more complete presentation.

\begin{thm}
Let $Z$ be a space with a bimonotone Schauder basis not having an $\ell_1$ spreading model.  Then $T_{G_Z,2}$
does not have any $\ell_p$ or $c_0$ as a spreading model. \label{noellporc0}
\end{thm}

Before passing to the proof we state two lemmas.

\begin{lem}
Suppose $y \in c_{00}$ and $\ee > 0$.  There is an $i_0 \in \nn$ such that for all $f \in D_{G_Z}$,
$f_{>i_0}(y) < \ee$. \label{finitesupport}
\end{lem}

\begin{proof}
Let $i_0$ such that $\sum_{i > i_0} |\supp ~y|/ m_i< \ee$.  The evaluation follows easily.
\end{proof}

The next lemma follows from standard arguments which, in the interest of brevity, we omit.

\begin{lem}
Let $f \in D_{G_Z}\setminus G_Z$ such that $w(f)=\{m_{j_0}\}$ for some $j_0 \in \nn$.  Let $j>j_0$ and $(x_i)_{i=1}^{n_j}$ be a normalized block sequence in $T_{G_Z,2}$. Then
$$f (\frac{1}{n_j} \sum_{i=1}^{n_j} x_{i} ) < \frac{3}{m_{j_0}} .$$
\label{smallweight}
\end{lem}

\begin{proof}[Proof of Theorem \ref{noellporc0}] It is easy to see
that niether $\ell_p$, for $1 < p <\infty$,
nor $c_0$ is are finitely block representable in $T_{G_Z.2}$ and therefore
can not be admitted as a spreading model.  Indeed, let $(y_k)_{k=1}^{\infty}$ be a block sequence in $T_{G,2}$.  For every $i \in \nn$ we have
$$\| \sum_{k=1}^{n_i} y_k \| \geq \frac{n_i}{m_i}.$$
We have assumed that for all $\alpha >0$, $\lim_{i \to \infty} n_i^\alpha/ m_i = \infty$ (assumption (c)).  Therefore for no $p>1$ does there
exist a $C_p$ such that for every $i \in \nn$,  $\| \sum_{k=1}^{n_i} y_k \| \leq C_p n_i^\frac{1}{p}$.

It remains to show that $T_{G_Z,2}$ does not admit an $\ell_1$ spreading model.  Let
$(w_n)$ be a bounded sequence generating an $\ell_1$ spreading model.   We must pass
to further subsequences of $(w_n)$ to achieve additional properties.   First,  it is well-known that since $(w_n) \subset X$ generates
an $\ell_1$-spreading model then for $0< \ee<10^{-4}$ we can find a block sequence $(y_n)$ of $(w_n)$
which generates a  $(1-\ee)$-$\ell_1$ spreading model.   Secondly, since $Y_{G_Z}$ does not admit
an $\ell_1$ spreading model, we may apply the Erdos-Madigor theorem
\cite{ErdosM} to find an $n_0 \in \nn$ and a block sequence $(z_n)$ of $(y_n)$ such that $z_n= \sum_{i \in F_n} x_i/n_0$ where
$|F_n|=n_0$ for all $n \in \nn$ and $\|z_n\|_{G_Z}<\ee$.  Passing to a further subsequence of $(z_n)_n$ (for example, $(z_{kn_0})_{k=1}^\infty$) we have a subsequence $(x_n)$ of $(z_n)$ satisfying
\begin{itemize}
\item$(x_n)$ generates and $\ell_1$ spreading model with constant $(1-\ee)$.
\item$\|x_n\|_{G_Z}< \ee$ for all $n \in \nn$.
\end{itemize}
The next step is to prove the following claim.

\begin{claim}There is an $i_0 \in \nn$ such that for each $n > 2$ there is a $\psi^n \in D_{G_Z}\setminus G_Z$ satisfying
\begin{itemize}
\item[(a)] $w(\psi^n) \leq m_{i_0}$;
\item[(b)] $\psi^n(x_n) > 1 - 4\sqrt{\ee}$.
\end{itemize}
\label{claim13}
\end{claim}
Since $(x_n)$ is a $(1-\ee)$-$\ell_1$ spreading model for each $n >2$ there is a $\phi^n= \sum_{i = 1}^k \lambda_i^n \phi^n_i $
such that $\phi^n(x_2 + x_n)> 2(1-\ee)$. It follows that
$$\phi^n(x_2)> 1- 2\ee \mbox{ and } \phi^n(x_n)> 1-2\ee.$$
Apply Lemma \ref{finitesupport} for $x_2$ and $\ee$ to find an $i_0$
such that for each $n \geq 2$, $\phi^n_{>i_0}(x_2) <\ee$.  We claim that
$\phi^n_{\leq i_0}$ is our desired $\psi^n$.  By definition $\phi^n_{\leq i_0}$
satisfies (a).  It suffices to prove that (b) holds. Notice that
\begin{equation}
\phi^n_{\leq i_0} (x_2) = \phi^n (x_2)- \phi^n_{>i_0} (x_2) >1-3\ee. \label{y2est}
\end{equation}
Now observe that
\begin{equation}
\phi^n_{\leq i_0} (x_n) > 1- 2 \ee - \sum_{i = i_0 +1}^k \lambda_i^n \phi^n_i (x_n). \label{lowersplitting}
\end{equation}
Using (\ref{y2est})
\begin{equation}
(\sum_{i=1}^{i_0} (\lambda_i^n)^2 )^\frac{1}{2} \geq \phi^n_{\leq i_0} (x_2) > 1- 3\ee.
\label{i}
\end{equation}
From (\ref{i}) we have,
\begin{equation}
 (\sum_{i=i_0+1}^\infty \lambda_i^n \phi^n_i(x_n) )^2 \leq  \sum_{i=i_0+1}^\infty (\lambda_i^n)^2 =
\sum_{i=1}^\infty (\lambda_i^n)^2 - \sum_{i=1}^{i_0} (\lambda_i^n)^2  < 3 \ee.\label{ii}
\end{equation}
Combining (\ref{lowersplitting}), (\ref{ii}) and the fact that $2\ee + \sqrt{3\ee} < 4\sqrt{\ee}$, (b) follows.

What (a) and (b) together tell us is that for {\em every} $n>2$ there is a functional $\psi^n$ which
almost norms $x_n$ and has only `small' (less than some fixed $m_{i_0}$) weights.  This allows
use to show, in the the next lemma, that no element in the sequence $(x_n)_n$ can be
normed by functionals with weights larger that $m_{i_0}$.
\begin{lem}
Let $n >2$ and $\phi \in D_{G_Z}$ with $w(\phi) > m_{i_0}$.  Then
$$\phi(x_n) < \frac{1}{2}. $$ \label{largeweight}
\end{lem}
\begin{proof} Let $\phi \in D_{G_Z}$ with $w(\phi)> m_{i_0}$.  Then $f= (\psi^n + \phi)/\sqrt{2} \in D_{G_Z}$.  Using Claim \ref{claim13} (b)
$$\phi(x_n) = \sqrt{2}f(x_n) - \psi^n(x_n) <  \sqrt{2} - 1 + (4\sqrt{\ee})<\frac{1}{2}.$$
As desired.
\end{proof}
We can now arrive at a contradiction using the following vector
$$z=\frac{1}{n_{i_0+1}} \sum_{q=1}^{n_{i_0+1}} x_{n_{i_0+1}+q}.$$
Find $\phi \in D_{G_Z}$ such that $\phi(z)>1-\ee$.
Using the Lemma \ref{smallweight} and Lemma \ref{largeweight} we have:
\begin{equation}
\begin{split}
.9< 1-\ee < \phi (z) & = \phi_{\leq i_0}(z)+ \phi_{>i_0}(z) < \sum_{i=1}^{i_0} \lambda_i \frac{3}{m_i} + \frac{1}{2} < \frac{3}{10} + \frac{1}{2}.
\end{split}
\end{equation}
This is a contradiction.
\end{proof}

We now describe the tree decomposition of the functionals in $D_{G_Z}$.  First we must set some notation.
Let $\nn^{<\nn}$ be the set of all finite tuples of $\nn$.  For $\delta, \gamma \in \nn^{<\nn}$ we
write $\delta \prec \gamma$ if $\delta$ is an initial segment of $\gamma$.  Let $\gamma(i)$ the the $i^{th}$ coordinate of $\gamma$.
Let $\nn^d$ denote the set of $d$-tuples of $\nn$ and $\nn^{\leq d}= \cup_{i \leq d} \nn^i$.  For $\gamma \in \nn^{d}$ let
$Im_{\gamma} \subset \nn^{d+1}$ denote the immediate successors of $\gamma$.
The following proposition describes a decomposition of the functionals in $D_{G_Z}$.  Tree decompositions
are a ubiquitous component in constructions of this type.  As such, we omit the proof of the proposition.

\begin{prop}
Let $n \in \nn$ and $f \in G_n \setminus G_0$.  Then there is a set $\mathcal{T}_f \subset \nn^{\leq 2n} \cup \{\emptyset \}$
and a collection $(f_\gamma)_{\gamma \in \mathcal{T}_f}$ of functionals which we call a {\em tree decomposition}
satisfying the following properties:
\begin{enumerate}
\item $f_{\emptyset} = f$.
\item Let $S_\gamma^f = Im_\gamma \cap \mathcal{T}_f$ and $\mathcal{T}_f^{d}= \mathcal{T}_f \cap \nn^{d}$.  If $\gamma \in \mathcal{T}_f$
and $S_\gamma^f =\emptyset$ we say that $\gamma$ is a terminal node.  In this case, $f_\gamma \in G_0$.
\item Let $0 \leq k <n$.  If $ \gamma \in \mathcal{T}_f^{(2k)}$ then
$$f_\gamma = \sum_{\delta \in S_\gamma^f} \lambda_\delta f_\delta \mbox{ where } w(f_\delta)=m_{\delta(2k+1)},~\sum_{\delta \in S_\gamma^f} \lambda_\delta^2 \leq 1$$
If $\gamma \in \mathcal{T}_f^{(2k+1)}$, then
$$f_\gamma = \frac{1}{m_{\gamma(2k+1)}}\sum_{\delta \in S_\gamma^f} f_\delta.$$
where $(f_\delta)_{\delta \in S_{\gamma}^f}$ are successive and $|S_{\gamma}^f| \leq n_{\gamma(2k+1)}$.
\end{enumerate}
\end{prop}

We need  one more definition.

\begin{defn}
Let $f \in D_{G_Z} \setminus G_Z$ and $\ttt_f \subset \nn^{<\nn} \cup \{\emptyset \}$ such that
the collection $(f_\gamma)_{\gamma \in \mathcal{T}_f}$ is a tree decomposition.
\begin{enumerate}
\item For $\alpha \in \ttt_f$ let $|\alpha|=k$ whenever $\alpha \in \nn^k$.
\item  Let $\mathcal{M}_f =\{ \alpha \in \ttt_f : \alpha \mbox{ is a terminal node of }\ttt_f\}$.
\end{enumerate}
\end{defn}

\section{$Z$ is a quotient of $T_{G_Z,2}$ }

As the title above suggests, the main objective of this section is to prove that $Z$ is a quotient of $T_{G_Z,2}$.  After
we establish this, we will proof the main theorem and one proposition.
To begin we require two lemmas.

\begin{lem}
Let $n \in \nn$ and $f \in D_{G_Z}$.  Suppose that for all $\alpha \in \aaa_f$, $|\alpha | \geq 2n$.
Then $\| f \|_\infty \leq 10^{-n}$. \label{SpirosLemma1}
\end{lem}

\begin{proof}
We proceed by induction on $n$.  For $n =1$ we have
\begin{equation*}
\| f \|_\infty = \|\sum_{\delta \in S_\emptyset } \frac{\lambda_\delta}{m_{\delta(1)}} \sum_{ \beta \in S_\delta} f_\beta \|_\infty  \leq \sum_{\delta \in S_\emptyset } \frac{1}{m_{\delta(1)}}  \sup_{ \beta \in S_\delta}  \| f_\beta \|_\infty < \frac{1}{10}.
\end{equation*}
In the above we used the for each $\delta \in S_\emptyset$, the functionals $(f_\beta)_{\beta \in S_\delta}$ have disjoint support.
Assume the claim for some $n \geq 1$.   We will prove it for $n+1$.
\begin{equation*}
\| f \|_\infty = \|\sum_{\delta \in S_\emptyset } \frac{\lambda_\delta}{m_{\delta(1)}} \sum_{ \beta \in S_\delta} f_\beta \|_\infty  \leq \sum_{\delta \in S_\emptyset } \frac{1}{m_{\delta(1)}}  \sup_{ \beta \in S_\delta}  \| f_\beta \|_\infty < \sum_{j=1}^\infty \frac{1}{m_j}\frac{1}{10^n} \leq \frac{1}{10^{n+1}}
\end{equation*}
In the above we used that for each $\delta \in S_\emptyset$, the functionals $(f_\beta)_{\beta \in S_\delta}$ have disjoint support and have terminal nodes each
of height greater than $2n$.
\end{proof}

\begin{lem}
Let $j_0 \in \nn$ and $f \in D_{G_Z}$ such that for all $\alpha \in \aaa_f$ there is a $\beta \prec \alpha$ such that $f_\beta$ is weighted and  $w(f_\beta) \geq m_{j_0}$.  Then $\|f\|_\infty \leq 2 \sum_{j \geq j_0} \frac{1}{m_j}$. \label{SpirosLemma2}
\end{lem}

\begin{proof}
For every $\alpha \in \aaa_f$ let
$$\beta_\alpha= \min\{\beta : \beta \prec \alpha,~ f_\beta \mbox{ is weighted and }w(f_\beta) \geq m_{j_0} \}.$$
Notice that if $\alpha \not=\alpha'$ are in $\aaa_f$ then $\beta_\alpha$ is either equal to or not comparable with $\beta_{\alpha'}$.
We will prove the following by induction:  For all  $\gamma \in \ttt_f $ such that there is an $\alpha \in \aaa_f$ with $\gamma \preceq \beta_\alpha$ one of the following holds:
\begin{enumerate}
\item If $\gamma = \beta_\alpha$ for some $\alpha \in \aaa_f$ then $\|f_\gamma \|_\infty \leq \frac{1}{w(f_\gamma)}.$
\item If $\gamma \prec \beta_\alpha$ for all $\alpha \in \aaa_f$ with $\gamma \prec \alpha$ and $f_\gamma$ is weighted then
$$\| f_\gamma \|_\infty \leq \frac{2}{w(f_\gamma)} \sum_{j \geq j_0} \frac{1}{m_j}.$$
\item If $\gamma \prec \beta_\alpha$ for all $\alpha \in \aaa_f$ with $\gamma \prec \alpha$ and $f_\gamma$ is not weighted
then
$$\| f_\gamma \|_\infty \leq 2 \sum_{j \geq j_0} \frac{1}{m_j}.$$
\end{enumerate}
After we prove the above, by taking $\gamma = \emptyset$, the lemma follows.

For the base case of the induction, we suppose that $\gamma= \beta_\alpha$ for some $\alpha \in \aaa_f$.  Since
it is clear that for all $\alpha \in \aaa_f$, $\|f_{\beta_\alpha} \|_\infty \leq 1/w(f_{\beta_\alpha})$, we are done.

Let $\gamma \in \ttt_f$ such that $\gamma \prec \beta_\alpha$ for all $\alpha \in \aaa_f$ with $\gamma \prec \alpha$.  Assume that for all $\tilde\gamma$ with
$\gamma \prec \tilde\gamma \preceq \beta_\alpha$ for some $\beta_\alpha$, (1), (2) or (3) holds (depending on $\tilde\gamma$).

Assume that $f_\gamma$ weighted.  Then
\begin{equation}
\|f_\gamma \|_\infty = \| \frac{1}{w(f_\gamma)} \sum_{\delta \in S_\gamma} f_\delta\|_\infty \leq \frac{1}{w(f_\gamma)} \max_{\delta \in S_\gamma} \|f_\delta\|_\infty \leq  \frac{2}{w(f_\gamma)}  \sum_{j \geq j_0} \frac{1}{m_j}
\end{equation}
In the above we used the induction hypothesis for $\delta \in S_\gamma$ since $\gamma \prec \delta \preceq \beta_\alpha$ whenever $\gamma \prec \beta_\alpha$.  Note
that if $\delta =\beta_\alpha$ then $\|f_\delta\|_\infty \leq 1/m_{j_0} < 2 \sum_{j \geq j_0} 1/m_j$.

Assume that $f_\gamma$ is not weighted.  Let $A_\gamma = \{ \delta \in S_\gamma : \delta = \beta_\alpha, ~\alpha \in \aaa_f\}$.  Splitting the set $S_\gamma$ and applying the induction
hypothesis we have
\begin{equation}
\begin{split}
\|f_\gamma \|_\infty & \leq \sum_{\delta \in S_\gamma}\| f_\delta\|_\infty =  \sum_{\delta \in A_\gamma }\| f_\delta\|_\infty +  \sum_{\delta \in S_\gamma \setminus A_\gamma }\| f_\delta\|_\infty \\
& \leq \sum_{j \geq j_0} \frac{1}{m_j} + \sum_{\delta \in S_\gamma \setminus A_\gamma } \frac{1}{w(f_\delta)} \sum_{j \geq j_0} \frac{1}{m_{j}} \leq 2 \sum_{j \geq j_0} \frac{1}{m_{j}}.
\end{split}
\end{equation}
In the above we used that $\sum_{\delta \in S_\gamma \setminus A_\gamma } \frac{1}{w(f_\delta)} < 1$.
\end{proof}

We are now ready to prove the main proposition of this section.

\begin{prop}
Let $Q: T_{G_Z,2} \to Z$ be the bounded linear map defined by $Q(e_i) = z_n$ for $i \in \Lambda_n$.  Then
$Q$ is a quotient map. \label{thequotient}
\end{prop}

Notice that $Q$ makes the same identifications as the map $Q_{G_Z}$ from Definition \ref{QZ}.

\begin{proof}
Let $z =\sum_{i=1}^d a_i z_i \in Z$ such that $\|z\|_Z =1$.   We will construct a vector $x$ such that $Qx=z$ and $\|x\|=1$; of course,
this is sufficient to prove the proposition.
Assume that $j_0 \in \nn$ satisfies the following:
\begin{enumerate}
\item $\sum_{j\geq j_0} \frac{2d}{m_j} < \frac{1}{5}$
\item $\frac{2((j_0-1) n_{j_0-1})^{s_{j_0}}}{n_{j_0}} < \frac{1}{5d}$
\item $\frac{d}{10^{s_{j_0}}}< \frac{1}{5}$
\end{enumerate}
For each $t \in\{ 1,\ldots , n_j\}$ let $(\ell_t^i)_{i=1}^d \subset \Lambda_i$ such that
$$\ell_t^1 < \ell_t^2 < \cdots < \ell_t^{d} < \ell_{t+1}^1 < \cdots. $$
Now set
$$x= \frac{1}{n_{j_0}} \sum_{t=1}^{n_{j_0}} \sum_{i=1}^d a_i e_{\ell_t^i} = \sum_{i=1}^d a_i \sum_{t=1}^{n_{j_0}} \frac{1}{n_{j_0}} e_{\ell_t^i}.$$
Let $y_t = \sum_i a_i e_{\ell_t^i}$.  Note that $(y_t)_{t=1}^{n_{j_0}}$ is a block sequence and $Qx=z$.
It is easy to see for all $t \in \{1, \ldots, n_{j_0}\}$,
\begin{equation}
\|y_t\|\leq \|y_t\|_1 \leq \sum_{i=1}^d |a_i| \leq d. \label{destimateyt}
\end{equation}
and
\begin{equation}
\|x\|\leq \|x\|_1 \leq \sum_{i=1}^d |a_i| \leq d. \label{destimate}
\end{equation}
We will also need the following easy remark
\begin{rem}
Let $g \in G_Z$ and $t \in \{1, \ldots , n_{j_0}\}$.  Then $g(y_t) \leq 1$. \label{groundonyt}
\end{rem}

Note that for all $t \in \{ 1, \ldots, n_{j_0}\}$,  $Q_{G_Z}(y_t)=z$.  Since $\|z\|=1$ we can apply
Proposition \ref{manyitems} (3) to deduce that $\|y_t\|_{G_Z} =1$.  The remark follows.

We observe first that $\|x\| \geq 1$: Suppose $(b_i)_{i=1}^d$ is a scalar sequence such that
$$ \sum_{i=1}^d a_i b_i = (\sum_{i=1}^d b_i z_i^*) ( \sum_{i=1}^d a_i z_i) =1.$$
By definition $\sum_{i=1}^d b_i (\sum_{t=1}^{n_{j_0}} e^*_{\ell_t^i}) \in G_Z$.  Thus
$$ \| x\| \leq \bigg(\sum_{i=1}^d b_i \sum_{t=1}^{n_{j_0}} e^*_{\ell_t^i}\bigg) \bigg(\sum_{i=1}^d a_i
 \frac{1}{n_{j_0}}\sum_{t=1}^{n_{j_0}} e^*_{\ell_t^i}\bigg)  = \sum_{i=1}^d a_i b_i =1.$$

Therefore, for $f \in D_G$ it suffices to show that $f(x) \leq 1$.  Partition $\aaa_f$ as follows:
\begin{equation*}
\begin{split}
& A_1 = \{ \alpha \in \aaa_f : | \alpha | \geq 2 s_{j_0}\} \\
& A_2 = \{ \alpha \in \aaa_f : \exists ~\beta \prec \alpha, ~w(f_\beta) \geq m_{j_0} \} \\
& A_3 = \aaa_f \setminus (A_1 \cup A_2)
\end{split}
\end{equation*}
Let $f= f_1 + f_2 +f_3$ such that for $i \in \{1,2,3\}$, $f_i$ has $A_i$ as its terminal nodes.  This splits
the rest of the proof naturally into three separate cases.  The first two cases are taken care of by Lemmas
\ref{SpirosLemma1} and \ref{SpirosLemma2} respectively.

Using Lemma \ref{SpirosLemma1}, (\ref{destimate}) and condition (3) on $j_0$
\begin{equation}
|f_1(x)| \leq \|f_1\|_\infty \|x\|_{1} \leq \frac{d}{10^{s_{j_0}}} <\frac{1}{5}. \label{f1}
\end{equation}
Similarly, using Lemma \ref{SpirosLemma2}, (\ref{destimate}) and condition (1) on $j_0$ we have
\begin{equation}
|f_2(x)| \leq \|f_2\|_\infty \|x\|_1 \leq 2d \sum_{j \geq j_0} \frac{1}{m_j} < \frac{1}{5}. \label{f2}
\end{equation}
To estimate $|f_3(x)|$ it is convenient to separate the support of $x$ into 2 sets.  Let
\begin{equation}
\begin{split}
E_2 =\{ t \in \{1, \ldots , n_{j_0} \} :  &\exists~\alpha \in A_3, ~\supp y_t \cap \range g_{\alpha} \not= \emptyset \mbox{ and } \\
& \supp y_t \not\subset \range g_{\alpha} \}
\end{split}
\end{equation}
and $E_1 = \{ 1, \ldots, n_{j_0} \} \setminus E_2$.

First we bound $|E_2|$ (the cardinality of $E_2$).
Observe that $|E_2| < 2|A_3|$.  Indeed, for each $t \in E_2$ there is an $\alpha \in A_3$ and each $\alpha \in A_3$
corresponds to at most $2$ elements of $E_2$.
By definition, for $\alpha \in A_3$: $|\alpha|< 2s_{j_0}$ and for all
$\beta \prec \alpha$ such that that $f_\beta$ is weighted, $w(f_\beta) \leq m_{j_0-1}$. These facts together yield
that $|A_3| \leq  ((j_0-1)n_{j_0-1})^{s_{j_0}}$.

Using the above along with condition (2) on $j_0$ we conclude that
\begin{equation}
|E_2| < 2|A_3| \leq 2((j_0-1) n_{j_0-1})^{s_{j_0}} <  n_{j_0}/(5d) \label{countingE2}
\end{equation}
Using (\ref{countingE2}) and (\ref{destimateyt}) we have
\begin{equation}
f_3(\frac{1}{n_{j_0}} \sum_{t \in E_2} y_t ) \leq \frac{1}{n_{j_0}} \sum_{t \in E_2} \|y_t\| \leq  \frac{d |E_2|}{n_{j_0}} <\frac{1}{5} \label{firstf3}
\end{equation}
We now pass to the final evaluation.  Let $x_1 = \sum_{t \in E_1} y_t$.  For $\gamma \in \ttt_{f_3}$
let
\begin{equation}
I_\gamma =\{ t \in \{ 1, \ldots , n_{j_0} \} : \supp y_t \subset \range f_\gamma \}.
\end{equation}
Let $\gamma \in \ttt_{f_3}$, we will prove the following:
\begin{enumerate}
\item If $\gamma \in A_3$ then $|f_\gamma( x_1)| \leq |I_\gamma|$.
\item If $\gamma \not \in A_3$ and $f_\gamma$ is weighted then
$$|f_\gamma (x_1)| \leq \frac{2}{w(f_\beta)} |I_\gamma|.$$
\item If $\gamma \not \in A_3$ and $f_\gamma$ is not weighted then
$$|f_\gamma (x_1)| \leq \frac{1}{5} |I_\gamma|.$$
\end{enumerate}
The proof goes by induction (and is similar to the proof of Lemma \ref{SpirosLemma2}).  For the base
case we assume that $\gamma \in A_3$.  Using Remark \ref{groundonyt} we have
$$|f_\gamma ( x_1)| \leq |f_\gamma ( \sum_{t \in I_\gamma} y_t)| \leq |I_\gamma|.$$
Assume that $\gamma \not \in A_3$ and that for all $\gamma'$ with $\gamma \prec \gamma'$
either (1), (2) or (3) holds (depending on $\gamma'$).

Our first case is when $f_\gamma$ is weighted.  Splitting the sum and applying the appropriate induction hypothesis we have
\begin{equation}
\begin{split}
|f_\gamma(x_1)|& \leq \frac{1}{w(f_\gamma)} \sum_{\delta \in S_\gamma\cap A_3} |f_\delta (x_1)| + \frac{1}{w(f_\gamma)} \sum_{\delta \in S_\gamma\setminus A_3} |f_\delta (x_1)| \\
& \leq \frac{1}{w(f_\gamma)}\bigg(  \sum_{\delta \in S_\gamma\cap A_3} |I_\delta| +  \frac{1}{5}\sum_{\delta \in S_\gamma\setminus A_3} |I_\delta|\bigg) \\
& \leq \frac{2}{w(f_\gamma)} |I_\gamma|
\end{split}
\end{equation}
Assuming $\gamma$ is not weighted, we again apply the induction hypothesis to get the desired estimate.
\begin{equation}
\begin{split}
|f_\gamma(x_1)| & \leq  \sum_{\delta \in S_\gamma} |f_\delta (x_1)| \leq \sum_{\delta \in S_\gamma} \frac{2}{w(f_\delta)} |I_\delta| \\
& \leq 2 \max_{\delta \in S_\gamma} |I_\delta| \sum^\infty_{j =1} \frac{1}{m_j} \leq \frac{1}{5} |I_\gamma|.
\end{split}
\end{equation}
The inductive proof is finished.  It follows that
\begin{equation}
|f_3 (\frac{1}{n_j} \sum_{t \in E_1} y_t )| = \frac{1}{n_j} |f_3(x_1)| \leq \frac{|E_1|}{5n_j} \leq \frac{1}{5}. \label{secondf3}
\end{equation}
Combining (\ref{f1}), (\ref{f2}), (\ref{firstf3}) and (\ref{secondf3}) we have
$$|f(x)| \leq |f_1(x)| + |f_2(x)| + |f_3(x)| < \frac{4}{5} <1.$$
This finishes the proof of the proposition.
\end{proof}

We can now prove our main theorem.  Of course, all that is required is to apply our previous work and compose quotient maps.

\begin{proof}[Proof of Theorem \ref{dichotomy}]
Let $X$ be a separable Banach space not admitting an $\ell_1$ spreading model.
By Proposition \ref{separabletobasis} the space $Z_X$ has a basis and does not admit an $\ell_1$
spreading model.  Moreover the map $Q_X: Z_X \to X$ is
a quotient map.  Let $Z_X=Z$.
Define $G_Z$ as in (\ref{groundZ}) and $T_{G_Z,2}$ as above.
Theorem \ref{noellporc0} says that $T_{G_Z,2}$ has no $\ell_p$ or $c_0$ spreading model.
Theorem \ref{thequotient} yields that the map $Q:T_{G_Z,2} \to Z$ is a quotient.
$Q_X \circ Q: T_{G_Z,2} \to X$ is the desired quotient.
\end{proof}

We conclude  with one last proposition that relates to our main theorem.  In particular, we note that there
does not exist a space $Y$ not admitting any $\ell_p$ or $c_0$ as a spreading model and
having, as a quotient, {\em every} space $X$ not admitting an $\ell_1$ spreading model.  In other words,
there is no {\em universal} space satisfying the requirements of our theorem.

\begin{prop}
Suppose $X$ has as a quotient every space not admitting an $\ell_1$ spreading model.  Then $X$ contains a copy of $\ell_1$.
\end{prop}

\begin{proof}
Recall that if the Bourgain $\ell_1$-index \cite{Bourgainell1}
of a space is unbounded (i.e. equals $\omega_1$) then the space contains $\ell_1$.
The main result of \cite{AMP} states that for each countable ordinal $\xi$ there is a separable space $X_\xi$ that
does not admit an $\ell_1$ spreading model and has hereditary $\ell_1$-index greater than $\omega^{\xi}$.  If a
space $X$ has, as quotient, every space not admitting an $\ell_1$ spreading model it must have
have the space $X_\xi$ as quotient for each $\xi<\omega_1$.  It follows that $X$ must have unbounded Bourgain index.

Looking more closely at the construction of $X_\xi$, one can observe that the {\em ground space} $X_{G_\xi}$ on which $X_\xi$ is built
also does not admit an $\ell_1$ spreading model and has $\ell_1$-index greater than $\omega^{\xi}$ (just not hereditarily).

Finally, we give the reader a concrete example:
Consider the following unconditional James tree space:
Let $J_{2,1}$ to be the completion of $c_{00}(\bt)$
equipped with the norm
\begin{equation} \label{5e15}
\|z\|=\sup\Big\{ \Big( \sum_{i=1}^d \big( \sum_{t\in\seg_i} |z(t)|\big)^{2} \Big)^{1/2} \Big\}
\end{equation}
where the above supremum is taken over all families $(\seg_i)_{i=1}^d$ of pairwise
incomparable non-empty segments of $\bt$.
For every well-founded tree $S$ of natural numbers, let $J^S_{2,1}$ be the closed subspace supported on the coordinates
of $S$.  Using arguments similar to those in \cite{AMP}, for every well-founded tree $S$, the space $J^S_{2,1}$ has no $\ell_1$ spreading
model.  It is easy to see that the Bourgain $\ell_1$ index of $J^S_{2,1}$ is at least the height of the well founded tree $S$.  Arguing as before, we conclude
that any space having each $J^S_{2,1}$ as a quotient must contain $\ell_1$.
\end{proof}

\section{Spaces not admitting quotients with separable duals}

In this section we answer affirmatively a problem posed in \cite[Remark VI]{JR}. 
The problem asks if there exists a separable Banach
space $X$ such that every infinite dimensional quotient has a non
separable dual. We note that the dual of such a space is
closely connected to HI spaces. Indeed, the dual $X^*$ must be non
separable and cannot contain $c_0$, $\ell_1$ or a reflexive subspace. 
Therefore, it does not contain a subspace with an
unconditional basis \cite{JamesUncond}. W. T.
Gowers' dichotomy \cite{Go1} yields that $X^*$ is saturated with HI spaces which do
not contain a reflexive subspaces. Next, we provide some
sufficient conditions for the 
existence of a space answering the
Johnson-Rosenthal question in the affirmative. We note that the sufficient conditions in 
the following theorem
are quite close to being necessary.

\begin{thm}
Let $X$ be a Banach space with the following properties:
\begin{enumerate}
\item $X$ does not contain a reflexive subspace.

\item $X^*$ is separable.

\item $X^{**}$ is hereditarily indecomposable.
\end{enumerate}
Then the dual $Y^*$ of any quotient $Y$ of $X^*$ is
non-separable.\label{theorem21}
\end{thm}

\begin{proof}
Assume on the contrary that  there exists a quotient $Y$ of $X^*$
with $Y^*$ separable. As it is shown in \cite{JR}, $Y$ has a
further quotient with a shrinking basis. Therefore, we assume that $Y$
has a shrinking basis $(y_n)_{n\in\mathbb{N}}$ and that the biorthogonal 
functionals $(y_n^*)_{n\in\mathbb{N}}$ form a boundedly complete basis of
$Y^*$, which is isomorphic to a subspace of $X^{**}$. It follows
that there exists a normalized boundedly complete basic sequence
$(w^{**}_n)_{n\in\mathbb{N}}$ in $X^{**}$. We will show that this
yields a contradiction. Indeed, since $X^{**}$ is HI, there exists
a normalized sequence $(z_n)_{n\in\mathbb{N}}$ in $X$ that is equivalent to a block
sequence of $(w^{**}_n)_{n\in\mathbb{N}}$; hence,
$(z_n)_{n\in\mathbb{N}}$ is also boundedly complete. Since $X^*$
is separable, the sequence $(z_n)_{n\in\mathbb{N}}$ has a further
block sequence $(v_n)_{n\in\mathbb{N}}$ which is normalized and
shrinking \cite{JR}. The sequence $(v_n)_{n\in\mathbb{N}}$ remains
boundedly complete and hence $Z =
\overline{<(v_k)_{k\in\mathbb{N}}>}$ is reflexive. This contradicts
assumption (i).
\end{proof}

\begin{cor}
There exists a separable Banach space $X$ such that every infinite
dimensional quotient has non separable dual.
\end{cor}

\begin{proof}
In \cite{AAT} a Banach space $Z$ is constructed satisfying the
assumptions of Theorem \ref{theorem21}. $Z^*$ is the desired
space.
\end{proof}

To conclude, we state the following problem that was communicated to the 
authors by W.B. Johnson.

\begin{ques}
Does every separable space have a quotient which is either HI or has an unconditional basis?
\end{ques}

This problem is a natural analogue of Gowers' dichotomy for quotients. In relation to this problem, V. Ferenczi \cite{FeQuotientDich} proved a dichotomy for quotients of subspaces of Banach spaces. In particular, we recommend section 3 of this paper which contains several interesting questions and observations relating to these types of problems.


\def\cprime{$'$} \def\cprime{$'$} \def\cprime{$'$}


\begin{thebibliography}{10}

\bibitem{AOST}
G.~Androulakis, E.~Odell, T.~Schlumprecht, and N.~Tomczak-Jaegermann.
\newblock On the structure of the spreading models of a {B}anach space.
\newblock {\em Canad. J. Math.}, 57(4):673--707, 2005.

\bibitem{AAT}
S.~A. Argyros, A.~D. Arvanitakis, and A.~G. Tolias.
\newblock Saturated extensions, the attractors method and hereditarily {J}ames
  tree spaces.
\newblock In {\em Methods in {B}anach space theory}, volume 337 of {\em London
  Math. Soc. Lecture Note Ser.}, pages 1--90. Cambridge Univ. Press, Cambridge,
  2006.

\bibitem{AKTkspreading}
S.~A. Argyros, V.~Kanellopoulos, and K.~Tyros.
\newblock Finite order spreading models.
\newblock preprint.

\bibitem{AMP}
S.~A. Argyros, A.~Manoussakis, and A.~M. Pelczar.
\newblock On the hereditary proximity to {$\ell_1$}.
\newblock http://arxiv.org/abs/0907.4317.

\bibitem{ATMemoirs}
S.~A. Argyros and A.~Tolias.
\newblock Methods in the theory of hereditarily indecomposable {B}anach spaces.
\newblock {\em Mem. Amer. Math. Soc.}, 170(806):vi+114, 2004.

\bibitem{BeauzamyLSM}
B.~Beauzamy and J.-T. Laprest{\'e}.
\newblock Mod\`eles \'etal\'es des espaces de {B}anach.
\newblock {\em Publ. D\'ep. Math. (Lyon) (N.S.)}, (4/A):ii+199+11, 1983.

\bibitem{Bourgainell1}
J.~Bourgain.
\newblock On separable {B}anach spaces, universal for all separable reflexive
  spaces.
\newblock {\em Proc. Amer. Math. Soc.}, 79(2):241--246, 1980.

\bibitem{ErdosM}
P.~Erd{\H{o}}s and M.~Magidor.
\newblock A note on regular methods of summability and the {B}anach-{S}aks
  property.
\newblock {\em Proc. Amer. Math. Soc.}, 59(2):232--234, 1976.

\bibitem{FeQuotientDich}
V.~Ferenczi.
\newblock A {B}anach space dichotomy theorem for quotients of subspaces.
\newblock {\em Studia Math.}, 180(2):111--131, 2007.

\bibitem{GowersTree}
W.~T. Gowers.
\newblock A {B}anach space not containing {$c\sb 0,\ l\sb 1$} or a reflexive
  subspace.
\newblock {\em Trans. Amer. Math. Soc.}, 344(1):407--420, 1994.

\bibitem{Go1}
W.~T. Gowers.
\newblock An infinite {R}amsey theorem and some {B}anach-space dichotomies.
\newblock {\em Ann. of Math. (2)}, 156(3):797--833, 2002.

\bibitem{JamesUncond}
R.~C. James.
\newblock Bases and reflexivity of {B}anach spaces.
\newblock {\em Ann. of Math. (2)}, 52:518--527, 1950.

\bibitem{JR}
W.~B. Johnson and H.~P. Rosenthal.
\newblock On {$\omega \sp{\ast} $}-basic sequences and their applications to
  the study of {B}anach spaces.
\newblock {\em Studia Math.}, 43:77--92, 1972.

\bibitem{OSKrivine}
E.~Odell and T.~Schlumprecht.
\newblock On the richness of the set of {$p$}'s in {K}rivine's theorem.
\newblock In {\em Geometric aspects of functional analysis ({I}srael,
  1992--1994)}, volume~77 of {\em Oper. Theory Adv. Appl.}, pages 177--198.
  Birkh\"auser, Basel, 1995.

\bibitem{RosenthalSM}
H.~P. Rosenthal.
\newblock Weakly independent sequences and the banach-saks property.
\newblock {\em Proc. of the Durham Symposium}, Juillet 1975.

\bibitem{Schechtman}
G.~Schechtman.
\newblock On {P}e\l czy\'nski's paper ``{U}niversal bases'' ({S}tudia {M}ath.
  {\bf 32} (1969), 247--268).
\newblock {\em Israel J. Math.}, 22(3-4):181--184, 1975.

\end{thebibliography}
\end{document}